\documentclass[11pt,bezier]{article}

\usepackage{amssymb,amsmath,graphicx,amssymb,amsfonts,wrapfig,color}
 \usepackage{amsthm}

\newtheorem{theorem}{Theorem}[section] 
\newtheorem{lemma}{Lemma}[section] 
 
\newtheorem{example}{Example}[section] 
\newtheorem{corollary}{Corollary}[section]

\date{}

\title{Some non-existence and asymptotic existence results for weighing matrices}

\author{
Ebrahim Ghaderpour $^{1}$ \\[0.3cm]
{\sl Lassonde School of  Engineering}\\
{\sl York University, Canada} }

\begin{document}\sloppy

\maketitle
\footnotetext[1]{E-mail: {\tt
ebig2@yorku.ca.}}

\begin{abstract}

Orthogonal designs and weighing matrices have many applications in areas such as coding theory, cryptography, wireless networking and communication. In this paper, we first show that if positive integer $k$ cannot be written as the sum of three integer squares, then there does not exist any skew-symmetric weighing matrix of order $4n$ and weight $k$, where $n$ is an odd positive integer. Then we show that for any square $k$, there is an integer $N(k)$ such that for each $n\ge N(k)$, there is a symmetric weighing matrix of order $n$ and weight $k$. Moreover, we improve some of the asymptotic existence results for weighing matrices obtained by Eades, Geramita and Seberry.

\end{abstract}

\section{Introduction}

An {\it orthogonal design} (OD) \cite[Chapter 1]{GS} of order $n$ and type $(s_1, \ldots, s_{\ell})$, denoted $OD(n; \ s_1, \ldots, s_{\ell})$, is a square matrix $X$ of order $n$ with entries from $\{0, \pm x_1, \ldots, \pm x_{\ell}\},$ where the $x_j$'s are commuting variables,  that satisfies
$$ XX^{\rm T}= \Bigg(\sum_{j=1}^{\ell} s_j x_j^2 \Bigg)I_n,$$ where $X^{\rm T}$ is the transpose of $X$, and $I_n$ is the identity matrix of order $n$.  An OD with no zero entry is called a full OD.
Equating all variables to $1$ in any full OD results in a Hadamard matrix.   
Equating all variables to $1$ in any OD of order $n$ results a weighing matrix, denoted $W(n,k)$, where $k$ is the weight that is the number of nonzero entries in each row (column) of the weighing matrix. 

 The {\it Kronecker product} of two matrices $A = [a_{ij}]$ and $B $ of orders  
 $m\times n$ and $r\times s,$ respectively, is denoted by $A \otimes B$, and it is the matrix of
order $mr \times ns$ defined by 
$$
 A \otimes B = \begin{bmatrix}
                a_{11}B & a_{12}B & \cdots & a_{1n}B\\
                a_{21}B & a_{22}B & \cdots & a_{2n}B\\
                \vdots  &  \vdots       &        & \vdots\\
                a_{m1}B & a_{m2}B & \cdots & a_{mn}B\\
               \end{bmatrix}\!.
 $$ 
The {\it direct sum} of $A$ and $B$ is denoted by $A\oplus B$, and it is the matrix of
order $(m+r) \times (n+s)$ which is defined as follows
$$
 A \oplus B = \begin{bmatrix}
                A & 0 \\
                0 & B \\
               \end{bmatrix}\!,
 $$ 
where $0$ represents a zero matrix of appropriate dimension. 

Let $A=(a_1, \ldots, a_{n})$ and $D=(d_1, \ldots, d_{n})$, where $a_i$'s and $b_i$'s belong to a commutative ring (e.g., $\mathbb{R}$). Square matrix $C=[c_{ij}]$ of order $n$ is called {\it circulant} if $c_{ij}=a_{j-i+1}$, where $j-i$ is reduced modulo $n$. Square matrix $B = [b_{ij}]$ of order $n$ is called {\it back-circulant} if $b_{ij} = d_{i+j-1}$, where $i+ j-2$ is reduced modulo $n$. We have $BC^{\rm T}=CB^{\rm T}$ and $B=B^{\rm T},$ i.e., any back-circulant matrix is symmetric. Let $R=[r_{ij}]$ be a square matrix of order $n$, where $r_{ij}=1$ if $i+j=1$ modulo $n$ and $r_{ij}=0$ otherwise. Matrix $R$ is called {\it back-diagonal} matrix. It is not hard to see that matrix $CR$ is back-circulant and so symmetric (cf., \cite[Chapter 4]{GS}). 

A {\it rational family} of order $n$ and type $(s_1 , \ldots , s_k),$ where the $s_i$'s are positive rational numbers, is a collection of $k$ rational matrices of order $n$, $A_1, \ldots, A_k,$ that satisfy
$A_iA_i^{\rm T}=s_iI_n$ ($1\leq i \leq k$) and $A_iA_j^{\rm T}=-A_jA_i^{\rm T}$, ($1\leq i\neq j \leq k$). A square matrix $A$ is skew-symmetric if $A^{\rm T}=-A$. Two matrices of the same dimension are {\it disjoint} if their entrywise multiplication is a zero matrix \cite[Chapters 1, 2]{GS}.

Eades and Seberry \cite[Chapter 7]{GS} showed some existence results for weighing
matrices. They showed that when the order of ODs and weighing matrices are {\it much larger} than the number of nonzero entries in each row, the necessary conditions for existence of ODs and weighing matrices are also sufficient. 
In this paper, we show some non-existence results on weighing matrices and some asymptotic  results for existence of weighing matrices.

\section{Non-existence results for weighing matrices}
In this section, we show some non-existence results for weighing matrices. The results are summarized in Theorems  \ref{nonodd}, \ref{nonskew} (known) and Theorem \ref{skewnon}.
\begin{lemma}{\rm (e.g., \cite[Chapter 2]{Ghaderpour}).}\label{realeigen}
The eigenvalues of a symmetric matrix with real entries are real.
\end{lemma}
\begin{lemma}{\rm (e.g., \cite[Chapter 2]{Ghaderpour}).}\label{ceigen}
 The eigenvalues of a skew-symmetric matrix with real entries are of the form $\pm ib,$ where b is a real number.
\end{lemma}
\begin{lemma}[e.g., \cite{GS}]\label{abseigen}
The absolute values of the eigenvalues of a weighing matrix $W(n,k)$ are $\sqrt{k}.$
\end{lemma}
\begin{theorem}[e.g., \cite{Craigenthesis}]\label{nonodd}
 There does not exist any symmetric weighing matrix with zero diagonal of odd order.
\end{theorem}
\begin{proof}
 Suppose that $W=W(n,k), n$ odd, is a symmetric weighing matrix with zero diagonal. 
From Linear Algebra, ${\rm tr}(W)=\sum_{t=1}^n\lambda_t,$ where $\lambda_t$'s are eigenvalues of $W.$  
By Lemma \ref{realeigen} and \ref{abseigen}, since $\lambda_t=\pm \sqrt{k},$
 $${\rm tr}(W)=\sum_{t=1}^n \lambda_t=c\sqrt{k}.$$
Since $n$ is odd, $c$ must be odd and therefore nonzero, but, by assumption, ${\rm tr}(W)=0$, which is a contradiction. 

\end{proof} 
\begin{theorem}[e.g., \cite{Craigenthesis}]\label{nonskew}
 There is no skew-symmetric weighing matrix of odd order.
\end{theorem}
\begin{proof}
 Assume that $W=W(n,k)$ is a skew-symmetric weighing matrix of odd order. From Lemma \ref{ceigen} and \ref{abseigen}, eigenvalues of $W$ are in form $\pm i\sqrt{k}.$ Therefore,
 $${\rm tr}(W)=\sum_{t=1}^n \lambda_t=ci\sqrt{k}.$$
Since $n$ is odd, $c$ must be odd and so nonzero, but since $W$ is skew-symmetric, ${\rm tr}(W)=0$ which is a contradiction.
\end{proof}

Next, we show that if $n$ is any odd number and $k$ cannot be written as the sum of three integer squares, then there is no skew-symmetric weighing matrix $W(4n,k)$. To do so, we first bring the following well-known results.  
\begin{lemma}[e.g., \cite{Serre}]\label{Gauss}
 A positive integer can be written as the sum of three integer squares if and only if it is not of the form $4^{\ell}(8k+7),$ where $\ell,k\geq 0.$
\end{lemma}
The following Lemma is a useful result that can be concluded from Lemma \ref{Gauss}, and for the sake of completion, we bring its proof.
\begin{lemma}[e.g., \cite{Serre}]\label{sumthree}
A positive integer is the sum of three rational squares if and only if it is the sum of three integer squares.
\end{lemma}
\begin{proof}
 Suppose that a positive integer $n$ is the sum of three rational squares. Reducing the three rational numbers to the same denominator,
one may write $$m^2n=\alpha^2+\beta^2+\gamma^2,$$ where $\alpha,$ $\beta$ and $\gamma$ are integers. Suppose that $n$ cannot be written as the sum of three integer squares. 
From Lemma \ref{Gauss}, there exist nonnegative integers $k,\ell$ such that  $n=4^{\ell}(8k+7).$ One may write $m$ as $2^r(2s+1),$ for some nonnegative integers $r,s.$ Thus, $m^2=4^r\big(4(s^2+s)+1\big)=4^r(8b+1)$, where $b=(s^2+s)/2$ is a nonnegative integer, and so
$$m^2n=4^{r+\ell}(8k+7)(8b+1)=4^{r+\ell}(8c+7),$$ 
where $c=8kb+k+7b$. This is a contradiction because by Lemma \ref{Gauss}, $m^2n$ cannot be written as the sum of three integer squares, whereas by assumption 
$m^2n=\alpha^2+\beta^2+\gamma^2.$  
Therefore, the result follows.
\end{proof}

\begin{lemma}[Shapiro \cite{Shapiro}]\label{red}
 There is a rational family in order $n=2^mt, \ t$ odd, of type $(s_1, \ldots, s_k)$ if and only if 
there is a rational family of the same type in order $2^m.$ 
\end{lemma}

\begin{lemma}[Geramita-Wallis \cite{GS}]\label{rat}
 A necessary and sufficient condition that there be a rational family of type $[1, k]$ in order $4$ is that $k$ be a sum of 
three rational squares.
\end{lemma}
\begin{proof}
 Suppose that $\{A,B\}$ is a rational family of type $[1, k]$ in order $4.$ Then \\ $\big\{I=A^{\rm T}A, \ D=A^{\rm T}B\big\}$ is also 
a rational family of the same type and order. Thus $D=-D^{\rm T}$ and $DD^{\rm T}=kI.$ Since $D$ is  skew-symmetric, the diagonal of $D$ is zero, so $k$ is a sum of three rational squares.

Now let $k=a^2+b^2+c^2,$ where $a,b$ and $c$ are rational numbers. If we let
$$D = \left[ \begin {array}{rrrr} 0&a&b&c \\
 -a&0&-c&b\\ -b&c&0&-a\\ -c&-b&a&0\end {array} \right]
,$$
then $\{I,D\}$ is a rational family of type $[1,k]$ and order $4.$
 \end{proof}
We use Lemmas \ref{sumthree}, \ref{red} and \ref{rat} to prove the following nonexistence result.
\begin{theorem}\label{skewnon}
 Suppose that positive integer $k$ cannot be written as the sum of three integer squares. Then there does not exist a skew-symmetric
$W(4n,k),$ for any odd number $n.$
\end{theorem}
\begin{proof}
If there is a skew-symmetric $W=W(4n,k)$ for some odd number $n,$ then   $\{I_{4n},W\}$ is a rational family of type $[1,k]$ and order $4n$. Thus, by Lemma \ref{red}, there is a rational family of type $[1,k]$ and order 4. Lemmas \ref{sumthree} and \ref{rat}  imply that $k$ must be the
sum of three integer squares. 
\end{proof}
\section{Asymptotic existence results of weighing matrices}

In this section, we provide some asymptotic results for existence of weighing matrices. These results are summarized in Theorem \ref{sym} and Theorems \ref{two} and \ref{four} (for ODs) that use different methodologies to improve the known results shown by Eades, Geramita and Seberry \cite[Chapter 7]{GS}.

\begin{lemma}[e.g., \cite{GS}]\label{OD}
 A necessary and sufficient condition that there exists an $OD\big(n; \ u_1, \ldots, u_k\big)$ is that there exists a family 
$\{A_1, \ldots, A_k\}$ of pairwise disjoint square matrices of order $n$ with  entries from $\{0, \pm 1\}$ satisfying

$(i)$ \ $A_i \ is \ a\ W(n,u_i),\ \ \ \ \ \  1\leq i \leq k, $

$(ii)$ $A_iA_j^{\rm T}=-A_jA_i^{\rm T}, \ \ \ \ \ \ \ \ 1\leq i\neq j \leq k. $
\end{lemma}
The following lemma, due to Sylvester, is known, and we bring its proof.
\begin{lemma}[\cite{Sylvester}]\label{ex}
 Let $x$ and $y$ be two relatively prime positive integers. Then every integer $N\geq xy$ can be written in the form $ax+by,$ 
where $a$ and $b$ are nonnegative integers.
\end{lemma}
\begin{proof}
 Let $N$ be an integer greater than or equal to $xy.$ Since $x$ and $y$ are relatively prime, there are integers $c$ and $d$ such that $cx+dy=N$ (see \cite{Serre}).
So, $$(c+jy)x+(d-jx)y=N,$$ where $j\in \Bbb{Z}.$ One can choose $j$ such that $0\leq c+jy\leq y-1.$ For such $j,$ we let 
$a=c+jy$ and $b=d-jx.$ The condition $N\geq xy$ implies that $b$ must be positive.  
\end{proof}
The following lemma shows how to construct ODs of higher orders by using two ODs of the same
types but different orders. The first part of the lemma is known (cf., \cite[Lemma 7.22]{GS}).
\begin{lemma} \label{twood}
Suppose that there are $A=OD\big(n_1;\ u_1, \ldots, u_m\big)$ and  $B=OD\big(n_2;\  u_1, \ldots, u_m\big).$ Let $h={\rm gcd}(n_1, n_2).$ Then there 
is an integer $N$ such that for each $t\geq N,$ there is an $OD\big(ht; \ u_1, \ldots, u_m\big).$ Moreover,  if $A$ and $B$ are symmetric, then there is an integer $N$ such that for each $t\geq N,$ there is a symmetric $OD\big(ht; \ u_1, \ldots, u_m\big)$.
\end{lemma}
\begin{proof}
 Let $x=n_1/h$ and $y=n_2/h.$ Then $x$ and $y$ are relatively prime. Let $N=xy,$ and $t$ be a positive integer
 $\geq N.$ By Lemma \ref{ex}, there are nonnegative integers $a$ and $b$ such that $t=ax+by.$ Since there exist $OD\big(n_1;\ u_1, \ldots, u_m\big)$ and  $OD\big(n_2; \ u_1, \ldots, u_m\big),$
there are families $\big\{A_1,\ldots, A_m\big\}$ of order $n_1$ and $\big\{B_1,\ldots, B_m\big\}$ of order $n_2$ satisfying the conditions
in Lemma \ref{OD}. We define the family $$S=\big\{(I_{a}\otimes A_1\oplus I_{b}\otimes B_1), \ldots, (I_{a}\otimes A_m\oplus I_{b}\otimes B_m)\big\}$$  of
order $an_1+bn_2=ht.$ It can be seen that this family satisfies the conditions of Lemma \ref{OD}, therefore it makes an $OD\big(ht; \ u_1, \ldots, u_m\big).$ \\
Now if $A$ and $B$ are symmetric, then $A_i$'s and $B_i$'s, $1\le i\le m$, are all symmetric. Since $$\big((A\otimes B)\oplus (C\otimes D)\big)^{\rm T}=(A^{\rm T}\otimes B^{\rm T})\oplus (C^{\rm T}\otimes D^{\rm T}),$$
set $S$ consists of $m$ symmetric matrices of order $ht$ satisfying the conditions of  Lemma \ref{OD}, and so they generate a symmetric  $OD\big(ht; \ u_1, \ldots, u_m\big).$
\end{proof}

\begin{theorem}[Seberry-Whiteman \cite{White}]\label{cirw}
 Let $q$ be a prime power. Then there is a circulant  $W\big(q^2+q+1, q^2\big).$
\end{theorem}
\begin{corollary}[\cite{Eades, GS}]\label{cw}
 Suppose that $q$ is a prime power and $c$ is any positive integer. Then there is a circulant $W\big(c(q^2+q+1), q^2\big).$
\end{corollary}
\begin{proof}
 Let $c$ be a fixed positive integer. From Theorem \ref{cirw}, we know that there exists a circulant $W\big(q^2+q+1, q^2\big).$ Suppose that the first row of this matrix is
$\big(a_1, a_2, \ldots, a_{q^2+q+1}\big).$ Let $$\phi(x)=\sum_{i=1}^{q^2+q+1}a_ix^i.$$ Thus, $\phi(\xi)\phi(\xi^{-1})= q^2,$ where $\xi$ is a primitive root of unity and $\xi^{q^2+q+1}=1.$
For $1\leq j\leq c(q^2+q+1)$ define
      $$b_j:=\left\{
      \begin{array}{l l}
      \displaystyle a_{\lceil \frac{j}{c}\rceil} \ \ \ j\equiv 1  {\pmod c}& \\
       0 \ \ \ \ \ \ \ {\rm otherwise}& 
       \end{array} \right.,$$
where $\lceil x\rceil$ is the smallest integer greater than or equal to $x.$
We show that if $W={\rm circ}\big(b_1, b_2, \ldots, b_{c(q^2+q+1)}\big)$, then $W$ is a $W\big(c(q^2+q+1), q^2\big).$ To see this, let 
$$\psi(y)=\sum_{j=1}^{c(q^2+q+1)}b_jy^j.$$
Thus we have $$\psi(y)=\sum_{i=1}^{q^2+q+1}a_iy^{c(i-1)+1}=y^{1-c}\sum_{i=1}^{q^2+q+1}a_iy^{ci}.$$ Since  $\phi(\xi)\phi(\xi^{-1})= q^2,$ 
for all $\xi$ such that $\xi^{q^2+q+1}=1,$ $\psi(\xi)\psi(\xi^{-1})=q^2,$ for all $\xi$ such that $\xi^{q^2+q+1}=1.$
Applying the finite Parseval relation $$\sum_{i=1}^{c(q^2+q+1)}b_ib_{i+r}=\frac{1}{c(q^2+q+1)}\sum_{j=1}^{c(q^2+q+1)}|\psi(\xi^j)|^2\xi^{jr},$$
where $i+r-1$ is reduced modulo $c(q^2+q+1)$, for $r=0$ gives $$\sum_{i=1}^{c(q^2+q+1)}b_i^2=\frac{1}{c(q^2+q+1)}\Big(c(q^2+q+1)q^2\Big)=q^2,$$
and for $1\leq r\leq c(q^2+q+1)-1$, $\sum_{i=1}^{c(q^2+q+1)}b_ib_{i+r}=0$. Therefore, $W$ is a circulant $W\big(c(q^2+q+1), q^2\big).$
\end{proof}
The next lemma shows how to make a symmetric OD to be used for Theorem \ref{sym}.
\begin{lemma}\label{symod}
 Let $k$ be a positive integer. Then there exists a symmetric $OD\big(2^k; \ 1_{(k)}\big)$. 
\end{lemma}
\begin{proof}
Define $A_1=\otimes_{m=1}^k  P$ and for $2\le n\le k$, $A_n=\otimes_{m=1}^{n-2}I\otimes Q\otimes_{m=n}^k P$, where $$P= \left[ \begin {array}{rr} 0&1\\ 1&0\end {array} \right]\!, \ \ Q= \left[ \begin {array}{rr} 1&0 \\ 0&-1\end {array} \right]\!, \ \ I= \left[ \begin {array}{rr} 1&0 \\ 0&1\end {array} \right]\!.$$
It can be directly verified that the family $\{A_1,\ldots , A_k\}$ of order $2^k$ satisfies the conditions of Lemma \ref{OD}, and therefore
it makes a symmetric $OD\big(2^k; \ 1_{(k)}\big)$. Note that $P, Q$ and $I$ are symmetric.
\end{proof}
\begin{theorem}[Robinson \cite{Robinson}]\label{R}
All $OD\big(2^t; \ 1,1,a,b,c\big)$ exist, where $t\geq 3$ and $a+b+c=2^t-2$.
 \end{theorem}

We prove the following well known lemma by giving a proof which is different from the proof in \cite[Lemma 7.27]{GS}. 
\begin{lemma}\label{e}
 For any sequence $\big(k_1, k_2, k_3,k_4\big)$ of positive integers, there is a positive integer $d$ such that there is a skew-symmetric $OD\big(2^d;\  k_1, k_2, k_3, k_4\big)$.
\end{lemma}
\begin{proof}
Let $t_1$ and $t_2$ be the smallest positive integers such that $1+k_1+k_2\le 2^{t_1}$ and $1+k_3+k_4\le 2^{t_2}$. By Theorem \ref{R}, there are 
$A=OD\big(2^{t_{1}}; \ 1,k_1,k_2\big)$ and $B=OD\big(2^{t_2}; \ 1,k_3,k_4\big).$ Without loss of generality, assume that
 $\big\{I_{2^{t_1}}, A_1,A_2\big\}$ and $\big\{I_{2^{t_2}}, B_1,B_2\big\}$ are two families corresponding to $A$ and $B$ satisfying the conditions
of  Lemma \ref{OD}. Let $P$ and $Q$ be the same matrices as in the proof of Lemma \ref{symod}. It can be directly verified
that the family $$\big\{I_{2^{t_2}}\otimes A_1 \otimes P, \ I_{2^{t_2}}\otimes A_2 \otimes P, \ B_1\otimes I_{2^{t_1}}\otimes Q
, \ B_2\otimes I_{2^{t_1}}\otimes Q \big\}$$ of four skew-symmetric matrices satisfies all conditions of
 Lemma \ref{OD}, and so it makes a skew-symmetric $OD\big(2^{t_1+t_2+1}; \ k_1, k_2, k_3, k_4\big).$ 
\end{proof}
\begin{corollary} [\cite{GS}]\label{g}
 Given any sequence $\big(k_1,k_2,k_3,k_4\big)$ of positive integers, there exists a positive integer $d$ such that there is an
$OD\big(2^d; \ 1, k_1, k_2, k_3, k_4\big).$ 
\end{corollary}

The following theorem, due to Geramita and Wallis, is known.
\begin{theorem}{\rm (Geramita and Wallis \cite[Theorem 7.14]{GS}).}\label{sym1}
 Suppose that $k$ is a square. Then there is an integer $N=N(k)$ such that for each $n\geq N$, there is a $W(n,k).$
\end{theorem}

We use a slightly different method to the proof of Theorem \ref{sym1} to give a proof of the following improved result.
\begin{theorem}\label{sym}
 Suppose that $k$ is a square. Then there is an integer $N=N(k)$ such that for each $n\geq N$, there is a symmetric $W(n,k).$
\end{theorem}
\begin{proof}
 Assume that $ k=\prod_{i=1}^m q_i^{2},$ where $q_i$ is either 1 or a prime power.  By Theorem \ref{cirw}, for each $i$ there exists a circulant $W_i=W\big(q_i^2+q_i+1, q_i^2\big).$
Let $$W=\otimes_{i=1}^{m} W_iR_i,$$ where $R_i$ is the back-diagonal matrix of order $q_i^2+q_i+1$. It can be seen that $W$ is a symmetric  $W\left(\prod_{i=1}^m(q_i^2+q_i+1), \prod_{i=1}^m q_i^{2}\right).$

Thus, there is an odd number $t=\prod_{i=1}^m(q_i^2+q_i+1)$ such that there is a symmetric $W(t,k).$ Moreover,
from Lemma \ref{symod}, there exists a symmetric $OD\big(2^k; 1_{(k)}\big)$ and so a symmetric $W(2^k,k).$ Now since $t$ is odd, 
${\rm gcd}(2^k,t)=1.$ Lemma \ref{twood} implies that there is a positive integer $N=N(k)$ such that for each $n\geq N$, there exists
a symmetric $W(n,k).$ 
\end{proof}

We prove the following theorem  by a slightly different method to the proof that first was given by  Eades \cite[Theorem 7.15]{Eades, GS}.
\begin{theorem}\label{two}
 Suppose that $k=k_1^2+k_2^2$, where $k_1$ and $k_2$ are two nonzero integers. Then there is an integer $N=N(k)$ such that for each $n\geq N$, there is an $OD\big(2n; \ k_1^2,k_2^2\big).$
\end{theorem}

\begin{proof}
For $j=1,2$, let $k_j^2=\prod_{i=1}^mq_{ij}^2,$ where $q_{ij}$ is either 1 or a prime power.
For each $i$, $1\le i \le m$, let $b_i={\rm lcm}\big\{q_{i1}^2+q_{i1}+1, q_{i2}^2+q_{i2}+1\big\}.$ From Corollary \ref{cw}, 
for each $j$, $j=1, 2$, and each $i$, $1\le i\le m$, there exists a circulant $W_{ij}=W\big(b_i, q_{ij}^2\big).$
It can be  seen that the following $2q\times 2q$ matrix  is an $OD\big(2q; \ k_1^2, k_2^2\big),$
$$\left[ \begin {array}{cc} 
x\displaystyle\otimes_{i=1}^{m}W_{i1}R_i&y\displaystyle\otimes_{i=1}^{m}W_{i2} \\ y\displaystyle\otimes_{i=1}^{m}W_{i2}&
-x\displaystyle\otimes_{i=1}^{m}W_{i1}R_i\end {array} \right]\!\!,
$$ where $R_i$ is the back-diagonal matrix of order $b_i, $ and
$q=\prod_{i=1}^m b_i$ is an odd number. 
From Theorem \ref{R}, one can choose the smallest positive integer $k$ such that there is an $OD\big(2^k; \ k_1^2, k_2^2\big).$ Since ${\rm gcd}\big(2q, 2^k\big)=2,$ Lemma \ref{twood} implies that
there is an integer $N=N(k)$ such that for each $n\geq N$, there is an $OD\big(2n;\ k_1^2, k_2^2\big).$
\end{proof}
Using the methodology in the proof of Theorem \ref{two}, the asymptotic bounds for the following two corollaries given by Eades \cite{Eades} are improved.
\begin{corollary}
 Suppose that  $k$ is the sum of two nonzero integer squares. Then there is an integer $N=N(k)$ such that for each $n\geq N$, there is a $W(2n, k).$
\end{corollary}
\begin{proof}
 Let $k=k_1^2+k_2^2,$ where $k_1$ and $k_2$ are integers. From Theorem \ref{two}, there is an integer $N=N(k)$ such that for any $n\geq N$, there is an $OD\big(2n;\ k_1^2,k_2^2\big),$ and 
so a $W(2n,k).$ 
\end{proof}

\begin{corollary}\label{2n11}
 Suppose that $d$ is an integer square. Then there exists an integer $N=N(d)$ such that for each $n\geq N$, there is a skew-symmetric $W(2n,d).$
\end{corollary}
\begin{proof}
 Suppose that $d=a^2.$ Let $k_1=1$ and $k_2=a.$ By Theorem \ref{two}, there exists an integer $N=N(d)$ such that for each $n\geq N$, there is an $OD\big(2n;\ 1,d\big)$,
 and so a skew-symmetric $W(2n, d).$ \end{proof}
We now use a different method to show Theorem \ref{four} shown by Eades \cite[Theorem 7.17]{Eades, GS} to improve the bounds ($N$) for the asymptotic existence of ODs of order $4n$, and consequently we prove Corollaries \ref{4n1}, \ref{4n2} and \ref{4n3}.
\begin{theorem}\label{four}
 Suppose that $k=k_1^2+k_2^2+k_3^2+k_4^2,$ where $k_1, k_2, k_3$ and $k_4$ are nonzero integers. Then there is an integer $N=N(k)$ such that for each $n\geq N$, there is an $OD\big(4n; \ k_1^2, k_2^2, k_3^2, k_4^2\big).$ 
\end{theorem}
\begin{proof}
Assume that $k=k_1^2+k_2^2+k_3^2+k_4^2$ and $k_1, k_2, k_3$ and $k_4$ are nonzero integers. Let $k_j^2=\prod_{i=1}^mq_{ij}^2,$ where $q_{ij}$ is either 1 or a prime power.
For each $i$, $1\le i\le m$, let $b_i={\rm lcm}\big\{q_{ij}^2+q_{ij}+1; \ \ j=1,2,3,4\big\}.$ From Corollary \ref{cw}, 
for each $j$, $1\le j\le 4$, and each $i$, $1\le i\le m$, there exists a circulant $W_{ij}=W(b_i, q_{ij}^2).$ Putting $$A=\displaystyle\otimes_{i=1}^{m}W_{i1}R_i,\ \ B=\otimes_{i=1}^{m}W_{i2},
 \ \ C=\otimes_{i=1}^{m}W_{i3}, \ \ D=\otimes_{i=1}^{m}W_{i4},$$ in the following array (Goethals-Seidel \cite{Go}) gives 
an  $OD\big(4q; \ k_1^2, k_2^2, k_3^2, k_4^2\big)$, 
$$\left[ \begin {array}{rrrr} xA&yB&zC&uD \\ -yB&xA&uD^{\rm T}&-zC^{\rm T} \\ -zC&-uD^{\rm T}&xA&yB^{\rm T} \\ -uD&zC^{\rm T}&-yB^{\rm T}&xA\end {array} \right]\!,$$
where  $q=\prod_{i=1}^m b_i$ which is an odd number, and $R_i$ is the back-diagonal matrix of order $b_i$. 

By Lemma \ref{e}, there is an $OD\big(2^d; \ k_1^2, k_2^2, k_3^2, k_4^2\big)$ for some suitable integer $d\ge 2$.
Since for $d\ge 2$, ${\rm gcd}(4q, 2^d)=4,$ Lemma \ref{twood} implies that there is an integer $N=N(k)$ such that for each $n\geq N$, there is an $OD\big(4n;\ k_1^2, k_2^2, k_3^2, k_4^2\big).$ Note that if some of the $k_i$'s are zero, then we consider the zero matrices.
\end{proof}

\begin{corollary}\label{4n1}
 Suppose that $d$ is any positive integer. Then there is an integer $N=N(d)$ such that for each $n\geq N$, there is a $W(4n,d).$
\end{corollary}
\begin{proof}
 It is a well known theorem of Lagrange \cite{element} that every positive integer can be written in the sum of four integer squares.  Let
$d=k_1^2+k_2^2+k_3^2+k_4^2. $ From Theorem \ref{four}, there is an integer $N=N(k)$ such that for each $n\geq N$, there is an $OD\big(4n;\ k_1^2, k_2^2, k_3^2, k_4^2\big)$, and
therefore a $W(4n,d).$
\end{proof}

\begin{corollary}\label{4n2}
 Suppose that $d$ is the sum of three integer squares. Then there exists an integer $N=N(d)$ such that for each $n\geq N$, there is a skew-symmetric $W(4n,d).$
\end{corollary}
\begin{proof}
 Consider $d=a^2+b^2+c^2,$ for some integers $a,b$ and $c.$ Substituting $k_1=a$, $k_2=b$, $k_3=c$ and $k_4=1$ in Theorem \ref{four}
gives the result. Note that the existence of an $OD(n;\ 1,h)$ is equivalent to existence of a skew-symmetric $W(n,h).$
\end{proof}
\begin{corollary}\label{4n3}
 Suppose that $d$ is any positive integer. Then there exists an integer $N=N(d)$ such that for each $n\geq N$, there is a skew-symmetric $W(8n,d).$
\end{corollary}
\begin{proof}
By Lagrange's  theorem \cite{vin}, one can write $d=k_1^2+k_2^2+k_3^2+k_4^2, $ where $k_i$'s are nonnegative integers. Let $A, B, C$ and $D$ be the same matrices as in Theorem \ref{four}. It can be seen that 
the following matrix gives an $OD\big(8q; \ 1, k_1^2, k_2^2, k_3^2, k_4^2\big),$ where $q$ is obtained as in Theorem \ref{four}, and is an odd number:
$$\left[ \begin {array}{cccc|cccc} xA&yB&zC&uD&wI_q&0&0&0 \\ -yB&xA&uD^{\rm T}&-zC^{\rm T}&0&wI_q&0&0 \\ -zC&-uD^{\rm T}&xA&yB^{\rm T}&0&0&wI_q&0 \\ -uD&zC^{\rm T}&-yB^{\rm T}&xA&0&0&0&wI_q \vspace{.1cm}
 \\ \hline \vspace{-.4cm}\\ wI_q&0&0&0&-xA&yB^{\rm T}&zC^{\rm T}&uD^{\rm T} \\ 0&wI_q&0&0&-yB^{\rm T}&-xA&uD&-zC \\ 0&0&wI_q&0&-zC^{\rm T}&-uD&-xA&yB \\ 0&0&0&wI_q&-uD^{\rm T}&zC&-yB&-xA\end {array} \right]\!.$$
From Corollary \ref{g}, there is an $OD\big(2^d;\ 1, k_1^2, k_2^2, k_3^2, k_4^2\big)$ for some suitable integer $d\ge 3.$
Since for $d\ge 3$, ${\rm gcd}(8q, 2^d)=8,$ Lemma \ref{twood} implies that there is an integer $N=N(d)$ such that for any $n\geq N$, there is an $OD\big(8n; \ 1, k_1^2, k_2^2, k_3^2, k_4^2\big),$ and so
a skew-symmetric $W(8n,d).$

\end{proof}

\begin{example}{\rm
Suppose that $k=92.$ Let $k_1=2, \ k_2=4, \ k_3=6$ and $ k_4=6$ in Theorem \ref{four}. Also, let 
$q_{11}=2, \ q_{21}=1, \ q_{12}=4, \ q_{22}=1, \ q_{13}=2, \ q_{23}=3, \ q_{14}=2$ and $ q_{24}=3.$ Then $b_1={\rm LCM}\{7,21,7,7\}=21,$
and $b_2={\rm LCM}\{3,3,13,13\}=39.$  By Theorem \ref{four}, there is an $$OD\big(4\cdot 21\cdot 39; \ 2^2,4^2,6^2,6^2\big).$$
From  Lemma \ref{e}, there is an $OD\big(2^{13}; \ 2^2,4^2,6^2,6^2\big)$. By Lemma \ref{twood}, since $h=\gcd(4\cdot 21\cdot 39, 2^{13})=4$, we have $N(92)\le 2^{11}\cdot 3^2\cdot 7\cdot 13$, and so for each
$n\geq N(92)$, there are a $W(4n,92)$ and a skew-symmetric $W(8n,92).$}
\end{example}

\section{Acknowledgement}
The paper constitutes Chapter $3$ of the author's Ph.D. thesis written under the direction of Professor Hadi Kharaghani at the University of Lethbridge. 
The author would like to thank Professor Hadi Kharaghani for introducing the problem and his very useful guidance toward solving the problem and also Professor Rob Craigen for his time and great help.

\end{document}